\documentclass[11pt]{amsart}
\usepackage{amsfonts, amsmath, amssymb, amscd, amsthm, graphicx, setspace, enumitem, color}
\usepackage{hyperref}
\newtheorem{thm}{Theorem}[section]
\newtheorem{cor}[thm]{Corollary}
\newtheorem{lem}[thm]{Lemma}
\newtheorem{prop}[thm]{Proposition}

\theoremstyle{definition}
\newtheorem{defn}[thm]{Definition}

\newtheorem{nt}[thm]{Notation}

\theoremstyle{remark}
\newtheorem{rem}[thm]{Remark}

\newcommand{\st }{{\rm st }}
\newcommand{\lk }{{\rm lk }}

\renewcommand{\ge }{\geqslant}

\begin{document}

\title[]{When is a TRAAG orderable?}

\address{Fac. Matem\'{a}ticas, Universidad Complutense de Madrid and  
Instituto de Ciencias Matem\'aticas, CSIC-UAM-UC3M-UCM,
Madrid, Spain}
\author{Yago Antol\'{i}n}
\email[Yago Antol\'{i}n]{yago.anpi@gmail.com}

\address{
Department of Mathematical Sciences, University of Copenhagen, Copenhagen, Denmark}
\author{Mart\'{i}n Blufstein}
\email[Mart\'{i}n Blufstein]{mabc@math.ku.dk}

\address{IMB, UMR 5584, CNRS, Universit\'e de Bourgogne, 21000 Dijon, France}
\author{Luis Paris}
\email[Luis Paris]{lparis@u-bourgogne.fr}

%
%
\begin{abstract}
We characterize, in terms of the defining graph, when a twisted right-angled Artin group (a group whose only relations among pairs of generators are either commuting or Klein-bottle type relations) is left-orderable. 
\end{abstract}
%
%

\maketitle



\section{Introduction}
Right angled Artin groups (RAAGs for short) form a family of groups defined through their presentation.
The presentation is codified by a simplicial graph,
where the vertex set corresponds to the generating set and the edge set corresponds to the relations. 
In this codification each edge represents a commuting relation among the vertices joined by an edge.
The concept of twisted Artin groups appears in \cite{ClancyEllis} and a normal form for twisted right angled Artin groups (TRAAGs for short) has been obtained by I. Foniqi \cite{Foniqi-thesis, Foniqi-paper}.
Again, TRAAGs form a family of groups defined through their presentation in which some of the relations among pairs of generators are Klein bottle relations of the form $aba^{-1} = b^{-1}$. 
As the generators have non-symmetric roles, the codification of the presentation is through a  mixed graph, which  is a simplicial graph in which some edges are oriented.

\begin{defn}
A {\it mixed graph} $\Gamma$ is a quintuple $(V,E,D,t,o)$ where $V$ is a non-empty set whose elements are called {\it vertices}; $E$ is a set of subsets of two elements of $V$ (i.e $E\subseteq\{\{x,y\}: x,y\in V, x\neq y\}$) and elements of $E$ are called {\it edges}; and $D$ is a subset of $E$ whose elements are called {\it oriented edges}. Finally, $o, t$ are functions $o,t\colon D\to V$, such that for all $e\in D$, $e=\{o(e),t(e)\}$.
\end{defn}

\begin{nt}
Given a mixed graph $\Gamma$ and $e\in D$ we will use  $[o(e),\tau(e)\rangle$ to denote $e$. If $e=\{a,b\}\in E-D$, we will use $[a,b]$ to denote  $e$.

It will be useful to make the following abuse of notation and also use  $[a,b]$ to denote the word $aba^{-1}b^{-1}$ and $[a,b\rangle$ to denote  the word $abab^{-1}$, so for $e\in E$ we might see it as a word on $V\cup V^{-1}$ or as an edge.
\end{nt} 

\begin{defn}
Let $\Gamma= (V,E,D,t,o)$ be a mixed graph. The {\it twisted right angled Artin group based on $\Gamma$} is the group presented by 
$$G_\Gamma = \langle V \mid  E\rangle.$$
Here we are making the abuse of notation mentioned above identifying each $e\in E$ with a word.
\end{defn}
\begin{rem}
In the above definition, if $D$ is empty then $G_\Gamma$ is a {\it right angled Artin group based on $\Gamma$}.
\end{rem}

Recall that a group $G$ is left-orderable (resp. bi-orderable) if there exists a total order on $G$ that is invariant under the left-multiplication (resp. left- and right-multiplication) action on $G$. 
Recall that
$$
 \; \text{bi-orderable} \; \Rightarrow 
 \; \text{left-orderable} \;\Rightarrow 
 \; \text{torsion-free}.
$$
Our main result characterizes when a TRAAG is torsion-free or left-/bi- orderable in terms of graph properties of a mixed graph, showing that TRAAGs might enjoy or fail to have any of these properties.

Before stating the theorem, we set some terminology. By an {\it oriented cycle} on a mixed graph $\Gamma = (V,E,D,o,t)$ we mean a sequence $e_0,\dots, e_n\in D$ such that $t(e_i)=o(e_{i+1})$ for $i<n$ and $t(e_n)=o(e_0)$. The oriented cycle  $e_0,\dots, e_n\in D$ is supported on a complete subgraph if for all $x,y\in \cup_{i=0}^n\{o(e_i),t(e_i)\}$ either $x=y$ or $\{x,y\}\in E$.

\begin{thm}\label{thm: main}
Let $\Gamma$ be a mixed graph and $G_\Gamma$ be the associated TRAAG.
The following holds
\begin{enumerate}
\item $G_\Gamma$ is torsion-free if and only if it contains no oriented cycle supported on a complete subgraph.
\item $G_\Gamma$ is  left-orderable if and only if it contains no oriented cycle.
\item $G_\Gamma$ is bi-orderable if and only if it contains no oriented edge.
\end{enumerate}
\end{thm}

Our main tool will be the normal form theorem for TRAAGs proved by I. Foniqi \cite{Foniqi-thesis,Foniqi-paper}.
We note that using the normal form theorem, Foniqi proved item (1) of the theorem. 
Also it follows from the normal form theorem (see Corollary \ref{cor: parabolic}) that if $U\subseteq V$ then $\langle U \rangle\sim G_\Delta$ where $G_\Delta$ is the TRAAGs defined on the full subgraph of $\Gamma$ spanned by $U$.
In particular, if $D$ is non-empty, then  the Klein bottle group $K=\langle a,b \mid aba^{-1}=b^{-1}\rangle$ embeds into the TRAAG, and hence it can not be bi-ordered as $K$ is not bi-orderable and bi-orderability passes to subgroups.
Also right-angled Artin groups are residually torsion-free nilpotent \cite{Droms} and hence bi-orderable. Therefore, item (3) of the theorem follows from well-known results.

The content of this note is proving item (2) of the theorem. 
We will show that in the case where $\Gamma$ is finite, we have that $G_\Gamma$ is poly-free if  $\Gamma$ does not contain an oriented cycle.
Recall that poly-free groups are locally indicable, that is, every finitely generated non-trivial subgroup has an infinite cyclic quotient. 
In particular, if $\Gamma$ is arbitrary, we will deduce that $G_\Gamma$ is locally indicable as any subgroup generated by a finite subset of $V$ is locally indicable. 
By the Burns-Hale theorem \cite{Burns-Hale} locally indicable groups are left-orderable.

From this discussion, it should be clear that Theorem \ref{thm: main} now follows from Proposition \ref{prop: cycle implies non-orderability} and Proposition \ref{prop: poly-free}. 
The first one states if the defining graph has an oriented cycle then the associated TRAAGs is not left-orderable and the second one states  that finite  mixed graphs without oriented cycles define poly-free TRAAGs.

We will recall the normal form theorem for TRAAGs and some other background material in Section \ref{section: background} and give the proofs of the mentioned propositions in Section \ref{section: proofs}. 

\section{Background}\label{section: background}
\subsection{Normal form on TRAAGs}
Given a (mixed) graph $\Gamma = (V,E, D, o,t)$ the {\it link} of $v\in V$ is the set of vertices $u\in V$ such that $\{u,v\}\in E$. 
It is denoted by $\lk(v)$.  
The {\it star} of $v$ is $\lk(v)\cup \{v\}$.

Let $\Gamma = (V,E, D, o,t)$ be a mixed graph and $G_\Gamma$ be the TRAAG associated to $\Gamma$.
Let $A= \cup_{v\in V} \langle v \rangle$. 
A word $v_1^{n_1}v_2^{n_2}\dots v_m^{n_m}\in A^*$ is {\it reduced } if it is empty (i.e. $m=0$), or if every $n_i\neq 0$ and for every $i<j$ such that $v_i=v_j$, $\{v_{i+1},\dots, v_{j-1}\}\not\subseteq \st(v_i)$.

\begin{thm}\cite[Theorem 3.4.13]{Foniqi-thesis}
Let $\Gamma$ be a mixed graph and $G_\Gamma$ be the associated group.
Every element $g\in G_\Gamma$ can be represented by a reduced word over $A= \bigcup_{v\in V}\langle v \rangle$, and a reduced word over $A$ represents the identity element of $G_\Gamma$ if and only if it is empty.
\end{thm}
 

We have the following corollary
\begin{cor}\label{cor: parabolic}
Let $U\subseteq V$ and $\Delta$ be the full subgraph of $\Gamma$ spanned by $U$.
Then $\langle U\rangle\cong G_\Delta$.
\end{cor}
\begin{proof}
There is a natural map $f\colon G_\Delta \to G_\Gamma$ given by $u\mapsto u$. Clearly, the image of $f$ is $\langle U\rangle$. Now, any reduced word in $G_\Delta$ is sent to a reduced word over $G_\Gamma$ by $f$ and hence $f$ is injective.
\end{proof}

\begin{nt}
Given $\Gamma=(V,E,D,o,t)$ a mixed graph and $U\subseteq V$ we will denote by $G_U$ the TRAAG based on the mixed full subgraph of $\Gamma$ spanned by $U$. 
Sometimes we will further abuse the notation and identify $G_U$ with $\langle U\rangle$. 
\end{nt}

\subsection{Left-orders on the fundamental group of a Klein bottle}
It is well-known fact that there are only 4 left-orders on $K=\langle a,b \mid aba^{-1}=b^{-1}\rangle$ and in all these orders the subgroup $\langle a \rangle$ dominates the left-order. We briefly recall the argument and make this statement precise. 

Recall that given a left-order $\prec$ on a group $G$, the set $P_\prec =\{g\in G \mid 1\prec g\}$ is a subsemigroup satisfying that $G$ is the disjoint union of $P,P^{-1}$ and $\{1\}$.  Conversely, any semigroup $P$ of $G$ that satisfies that  $G$ is the disjoint union of $P,P^{-1}$ and $\{1\}$ defines a left-order by $x\prec_P y\colon \Leftrightarrow x^{-1}y\in P$. 
Such subsemigroups are called {\it positive cones}.

It is clear that $\mathbb{Z}$ only has two possible positive cones (namely $\mathbb{Z}_{>0}$ and  $\mathbb{Z}_{<0}$) and hence it has only two left-orders.
In the case $K$ with the above presentation, for $\varepsilon, \mu \in \{ \pm 1\}$ we set $P_{\varepsilon, \mu} = \{ a^{\varepsilon n} b^{\mu m} \mid n \ge 1 \text{ and } m \in \mathbb Z, \text{ or } n=0 \text{ and } m \ge 1\}$. 
Then we have the following
\begin{lem}
The subsemigroups 
$P_{1,1}$, $P_{1,-1}$, $P_{-1,1}$, $P_{-1,-1}$
are positive cones.
\end{lem}

The proof of the lemma is elementary and well-known. 
We exemplify the situation briefly and the reader might fill the details.
For example, Using the relations $b^{-1}a =ab $, $ba=ab^{-1}$ one sees that every element can be written as $a^nb^m$ and that $P = P_{1,1} = \{ a^n b^m \mid n \ge 1 \text{ and } m \in \mathbb Z, \text{ or } n=0 \text{ and } m \ge 1\}$ is a subsemigroup.
It is now easy to see that $P, P^{-1}, \{1\}$ are disjoint and the union is $K$.

Notice that the previous lemma implies that there are only 4 possible left-orders and they are determined by which of $a,a^{-1}$ and $b,b^{-1}$ is positive.

From the exact sequence $1\to \langle b \rangle \to K\stackrel{f}{\to} \langle a \rangle \to 1$ we see how to construct 4 left-orders on $K$ from each possible left-order on $\langle a \rangle$ and $\langle b \rangle$.
Fixing the left-orders $\prec_{a}$ on $\langle a \rangle$ and $\prec_b$ on $\langle b\rangle$ we set $g\prec h$ if $f(g)\prec_a f(h)$ or $f(g)=f(h)$ and $1\prec_b g^{-1}h$.

This shows that 
\begin{lem}\label{lem: convex}
 For any left-order $\prec$ on $K=\langle a,b \mid aba^{-1}b\rangle $ one has that $\langle b \rangle \prec \max_\prec\{a,a^{-1}\}$.
\end{lem}

\begin{rem}\label{rem: Tararin}
If fact the previous lemma is much more general and for any left-order $\prec$ on $BS(1,-m)=\langle a,b \mid aba^{-1} =b^{-m}\rangle$ $(m>0)$ one has that $\langle b \rangle \prec \max_\prec\{a,a^{-1}\}$. 
Those are examples of Tararin \cite{Tararin} groups, which are groups with finitely many left-orders and  all the left-orders are understood.
This fact follows easily from the theory of Tararin groups but it will be not needed  for our purposes. See alternatively \cite[Theorem 2.2.13]{GOD}. 
\end{rem}

\section{Proofs}\label{section: proofs}

\begin{prop}\label{prop: cycle implies non-orderability}
Let $\Gamma$ be a mixed graph and $G_\Gamma$ be the associated TRAAG. 
If $\Gamma$ contains an oriented cycle, then $G_\Gamma$ is not left-orderable.
\end{prop}
\begin{proof}
Let $e_0=[x_0,x_1\rangle, e_1=[x_1,x_2\rangle, \dots e_n=[x_n,x_0\rangle$ be an oriented cycle. 
We consider the indices of the $x_i$ mod $n+1$. 
Therefore for each $i$, we have $x_i^{-1} = x_{i+1}x_ix_{i+1}^{-1}$.
Recall $\langle x_i,x_{i+1}\rangle$ is naturally isomorphic to $K=\langle a,b \mid aba^{-1} = b^{-1}\rangle$ via $x_i\mapsto b, x_{i+1}\mapsto a$.
Suppose that $G_\Gamma$ is left-orderable with left-order $\prec$.
Then, by Lemma \ref{lem: convex}, we have that
$\max_{\prec}\{x_i, x_i^{-1}\}\prec \max_{\prec}\{x_{i+1},x_{i+1}^{-1}\}$ for all $i$. This implies that $\max_{\prec}\{x_0, x_0^{-1}\}\prec \max_{\prec}\{x_0, x_0^{-1}\}$ which is a contradiction.
\end{proof}
\begin{rem}
Following Remark \ref{rem: Tararin} one sees that the same argument shows that $$G=\langle x_0,x_1,x_2,x_3 \mid  x_0x_1x_0^{-1}=x_1^{-2}, x_1x_2x_1^{-1}= x_2^{-2}, x_2x_3x_2^{-1} = x_3^{-2}, x_3x_0x_3^{-1}= x_0^{-2}\rangle$$
is not left-orderable. 

For that it is enough to show that each $\langle x_i,x_{i+1}\rangle$ (indexes mod $4$) is isomorphic to $BS(1,-2)=\langle a,b \mid aba^{-1} = b^{-2}\rangle$.
Recall that in this group both $a$ and $b$ are non-trivial and of infinite order.
Observe that $\langle a,b,c \mid aba^{-1}=b^{-2}, bcb^{-1}=c^{-2}\rangle$ is an amalgamated free product of two $BS(1,-2)$ along an infinite cyclic subgroup. 
From the action on the Bass-Serre tree it is easy to deduce that $\{ a, c\}$ freely generates a free group.
Hence $G$ is an amalgamated free product $\langle x_0,x_1,x_2\rangle *_{\langle x_0,x_2\rangle} \langle x_2,x_3,x_0\rangle$.
It follows now that each $\langle x_i,x_{i+1}\rangle$ is isomorphic to $BS(1,-2)$ via $x_i\mapsto a, x_{i+1}\mapsto b$. 

Other examples can be constructed in this vein. This particular example contrasts  with the non-trivial fact that the Higman group
$$H=\langle x_0,x_1,x_2,x_3 \mid  x_0x_1x_0^{-1}=x_1^{2}, x_1x_2x_1^{-1}= x_2^{2}, x_2x_3x_2^{-1} = x_3^{2}, x_3x_0x_3^{-1}= x_0^{2}\rangle$$
is left-orderable.
This was proved by C. Rivas and M. Triestino \cite{RivasTriestino}.
\end{rem}

Given a mixed graph, a vertex $v$ is a {\it source}  if for all $e\in E$ such that $v\in E$ one has that either $e=[u,v]$ or $e=[v,u\rangle$.

\begin{lem}\label{lem: source}
Let $\Gamma$ be a mixed graph and $G_\Gamma$ be the associated TRAAG. 
Let $v$ be a source and and $U=V-\{v\}$.
Then the homomorphism  $\rho \colon G_\Gamma \to G_U$ given by $u\mapsto u$ for all $u\in U$ and $v\mapsto 1$ is well defined and the kernel is free. 
\end{lem}
\begin{proof}
Suppose that $v\in V$ is a source. 
Then any $u\in U\colon = V-\{u\}$ either has no relation with $v$ or $u$ conjugates $v$ to either $v$ or $v^{-1}$.
This means that the image of any relation after killing $v$ is of the form $uu^{-1}$ and hence the function $\rho$ of the statement is well-defined.

It is clear from the presentation that $G_\Gamma = G_{\st(v)}*_{G_{\lk(v)}} G_U$.
Let $T$ be the Bass-Serre tree associated to the amalgamated free product decomposition above.
The $G_\Gamma$-tree $T$ has one $G_\Gamma$-orbit of edges, whose stabilizers are conjugated to $G_{\lk(v)}$ and two $G_\Gamma$-orbits of vertices, whose stabilizers are either conjugates to $G_{\st(v)}$ or $G_{U}$. 
Let $N=\ker \rho$. 
Note that for any $g\in G_\Gamma$, $N\cap gG_U g^{-1}$ is trivial, as the restriction of $\rho$ to $G_U$ is the identity.
Therefore, $N$ acts freely on the edges of $T$ and on vertices of the orbit with $G_\Gamma$-stabilizer conjugated to $G_U$.

Using the relations $uv=vu$ or $uv=v^{-1}u$, it is easy to see that $G_{\st(v)}=\langle v\rangle \cdot G_{\lk(v)}$. 
Therefore $N\cap G_{\st(v)}=\langle v\rangle$ as $\rho$ is injective restricted to $G_{\lk(v)}$. 
As $N$ is normal, $N\cap gG_{\st(v)}g^{-1}$ is cyclic for any $g\in G_\Gamma$ .
Therefore $N$ acts on a tree with trivial edge stabilizers and infinite cyclic or trivial vertex stabilizers. 
This implies, $N$ is the fundamental group of a graph of groups with trivial edge groups and infinite cyclic or trivial vertex groups, and therefore $N$ is free.
\end{proof}

Recall $G$ is {\it poly-free} if there exists a finite sequence
$$G_0=\{1\} \unlhd G_1 \unlhd G_2 \unlhd \dots \unlhd G_n= G$$
such that $G_{i+1}/G_i$ is free.
It follows from the definition that if $N\unlhd G$ is free and $G/N$ is poly-free, then $G$ is poly-free.

\begin{prop}\label{prop: poly-free}
Let $\Gamma$ be a finite mixed graph without an oriented cycle, then the group $G_\Gamma$ is poly-free. 
\end{prop}
\begin{proof}
The proof is by induction on $|V|$. 
If $|V|=1$ the group is infinite cyclic and hence poly-free.

Suppose that $|V|>1$ and there exists $v\in V$ that is a source. 
Then, by Lemma \ref{lem: source}, the normal closure $N$ of $v$ in $G_\Gamma$ is free and $G_\Gamma/K\cong \langle V-\{v\} \rangle \cong G_{U}$ where $U=V-\{v\}$.
Note that as $G_\Gamma$ does not contain an oriented cycle neither does $G_\Delta$. 
By induction, $G_U$ is poly-free and hence $G_U$ is poly-free.

So it remains to show that a source exists. 
Suppose, for the sake of finding a contradiction, that no vertex is a source.
This means that for every vertex $v$, there is $u\in V$ such that $[u,v\rangle\in D$. 
Let $v_0\in V$ and for $i>0$ choose  $v_i$  a vertex such that $[v_{i-1}, v_i\rangle \in D$. As $\Gamma$ is finite, there must $i\neq j$ such that $v_i=v_j$ and hence we found an oriented cycle. 
This is the desired contradiction.
\end{proof}

\noindent {\bf Acknowledgements}. The content of this note originated when the three authors met at the University of Seville in a course organized by María Cumplido. We thank her and the University of Seville for their hospitality.
Y. Antolín acknowledges partial support by the grant PID2021-126254NB-I00 of the Ministry of Science and Innovation of Spain.
M. A. Blufstein acknowledges support from the European Union’s Horizon 2020 research and innovation programme under the Marie Skłodowska-Curie grant agreement No 777822.


\end{document}